\documentclass[12pt,a4paper]{amsart}
\usepackage{amssymb}

\usepackage{graphicx,amssymb,amsfonts,epsfig,amsthm,a4,amsmath,url}
\usepackage[latin1]{inputenc}

\newtheorem{Thm}{Theorem}[section]
\newtheorem{Prop}[Thm]{Proposition}
\newtheorem{Lem}[Thm]{Lemma}
\newtheorem{Cor}[Thm]{Corollary}
\theoremstyle{definition}
\newtheorem{Def}[Thm]{Definition}

\newtheorem{Rem}[Thm]{Remark}
\newtheorem{Ex}[Thm]{Example}

\newcommand{\Z}{\mathbf{Z}}

\newcommand{\R}{\mathbf{R}}
\newcommand{\C}{\mathbf{C}}
\newcommand{\Q}{\mathbf{Q}}
\newcommand{\K}{\mathbf{K}}
\newcommand{\T}{\mathbf{T}}

\newcommand{\GL}{\textnormal{GL}}
\newcommand{\SL}{\textnormal{SL}}
\newcommand{\PSL}{\textnormal{PSL}}

\newcommand{\PGL}{\textnormal{PGL}}

\newcommand{\HH}{\mathcal{H}}

\newcommand{\eps}{\varepsilon}

\title{The Howe-Moore property for real and $p$-adic groups}

\author[Cluckers]{Raf Cluckers}
\address{Universit\'e Lille 1, Laboratoire Painlev\'e, CNRS - UMR 8524,
Cit\'e Scientifique, F-59655
Villeneuve d'Ascq Cedex, France; and
Katholieke Universiteit Leuven, Department of Mathema-tics,
Celestijnenlaan 200B, B-3001 Leuven, Belgium}
\email{raf.cluckers@wis.kuleuven.be}

\author[Cornulier]{Yves Cornulier}
\address{IRMAR, Campus de Beaulieu, 35042 Rennes CEDEX, France}
\email{yves.decornulier@univ-rennes1.fr}

\author[Louvet]{Nicolas Louvet}
\address{Institut de Mathématiques - Université de Neuchâtel,
Rue Emile Argand 11, CH-2007 Neuchâtel - Switzerland}
\email{nicolas.louvet@unine.ch}

\author[Tessera]{Romain Tessera}
\address{UMPA, ENS de Lyon,
46, allée d'Italie, F-69364 Lyon Cedex 07 - France}
\email{rtessera@umpa.ens-lyon.fr}

\author[Valette]{Alain Valette}
\address{Institut de Mathématiques - Université de Neuchâtel,
Rue Emile Argand 11, CH-2007 Neuchâtel - Switzerland}
\email{alain.valette@unine.ch}

\date{\today}

\begin{document}

\baselineskip=16pt

\maketitle

\begin{abstract}
We consider in this paper a relative version of the Howe-Moore Property, about vanishing at infinity of coefficients of unitary representations. We characterize this property in terms of ergodic measure-preserving actions. We also characterize, for linear Lie groups or $p$-adic Lie groups, the pairs with the relative Howe-Moore Property with respect to a closed, normal subgroup. This involves, in one direction, structural results on locally compact groups all of whose proper closed characteristic subgroups are compact, and, in the other direction, some results about the vanishing at infinity of oscillatory integrals.
\end{abstract}

\section{Introduction}

A locally compact group $G$ has the {\it Howe-Moore property} if any
unitary representation $\pi$ of $G$ without non-zero fixed vector is
a $C_{0}$-representation, i.e. the coefficients of $\pi$ vanish at
infinity on $G$. It is a basic result of Howe and Moore
\cite{HowMoo} that a connected, simple real Lie group with finite center, has this property. More generally, so does the subgroup generated by unipotent elements in a simple
algebraic group over a local field. Other examples of Howe-Moore groups
include the bicolored automorphism groups of regular or biregular trees
of bounded degree, see \cite{LubMoz}. The Howe-Moore property has
applications to ergodic theory, since every measure-preserving ergodic action of a non-compact Howe-Moore group on a probability space, is necessarily mixing (see \cite{BeMa,Zim}). We first observe that, for second countable groups, the converse is true:

\begin{Prop}\label{Prop:mixing} Let $G$ be a second countable, locally compact group. The group $G$ has the Howe-Moore property if and only if every ergodic, measure preserving action of $G$ on a probability space, is mixing.
\end{Prop}

The Howe-Moore property imposes quite stringent algebraic conditions on the ambient group (e.g. every proper, closed normal subgroup is compact, see Proposition \ref{perm2} below). Using this observation, we prove two structural results on Howe-Moore groups. The first holds for arbitrary locally compact groups. Say that a topological group is {\it locally elliptic} if every finite subset topologically generates a compact subgroup.

\begin{Prop}\label{loccompHM} Let $G$ be a non-compact, locally compact group with the Howe-Moore property. Then either $G$ is locally elliptic and not compactly generated, or there exists a compact normal subgroup $K\vartriangleleft G$ such that $G/K$ is topologically simple.
\end{Prop}

Proposition \ref{loccompHM} gives a strong restriction on the class of groups with the Howe-Moore Property. To obtain more examples, we introduce a relative version of the Howe-Moore property.

\begin{Def} Let $H$ be a closed subgroup of the locally compact group $G$. The pair $(G,H)$ has the {\it relative Howe-Moore property} if every representation $\pi$ of $G$ either has $H$-invariant vectors, or the restriction $\pi|_H$ is a $C_0$-representation of $H$.
\end{Def}

We characterize such pairs in the case of groups having a faithful linear representation over a local field of characteristic zero; since any such field is a finite extension of $\R$ or $\Q_p$, it is enough to consider these two fields.

\begin{Thm}Let $\K$ denote $\R$ or $\Q_p$. Let $G$ be a locally compact group isomorphic to a closed subgroup of $\GL_n(\K)$. Let $N$ be a non-compact closed normal subgroup of $G$. Then $(G,N)$ has the relative Howe-Moore Property if and only if one of the following conditions holds.
\begin{itemize}
\item[(1)] $N\simeq\K^n$ ($n\ge 1$), and the representation of any open subgroup of $G$ by conjugation on $N$ is irreducible and non-trivial.
\item[(2)] $N\simeq S^+$, the subgroup generated by unipotent elements in $\mathcal{S}(\mathbf{K})$, where $\mathcal{S}$ is a simple algebraic $\K$-group.
\end{itemize}\label{main}
\end{Thm}
Note that if $\K=\R$, (1) amounts to say that the connected component of identity $G_0$ acts irreducibly on $N$.

Our proof of the sufficiency in Theorem \ref{main} is based on the Mackey machine for semidirect products and vanishing at infinity of Fourier transforms of singular measures on $\K^n$.

A motivation for the paper was the study of Property (BP): as B.~Bekka pointed out to us, if $(G,N)$ has the relative Howe-Moore Property, then it has the relative Property (BP), i.e. for every affine isometric action of $G$ on a Hilbert space, its restriction to $N$ is either proper or has a fixed point. We plan to come back to property (BP) in a subsequent paper.

\medskip

The paper is organized as follows: Section \ref{sec:2} discusses various characterizations of the Howe-Moore property (and its relative version), Section \ref{sec:3} treats structural consequences of these properties. Section \ref{sec:4} contains the main argument that we need to prove the relative Howe-Moore Property, in the essential case of a semidirect product with an abelian group. Sections \ref{sec:5} and \ref{sec:6} mainly concern the structure of Lie groups and $p$-adic Lie groups, giving restriction on the possible pairs with the relative Howe-Moore Property, resulting in Theorem \ref{main}. On our way, we are led to completely characterize analytic real (in Section 5) and $p$-adic groups (in Section 6) such that every proper, closed characteristic subgroup is compact.

{\bf Acknowledgements}: 
We thank J. Ludwig for pointing out the reference \cite{Stein} on decay properties of Fourier transforms on singular measures on $\R^n$. Special thanks are due to P.-E. Caprace for considerable input at various stages, in particular for suggesting Proposition \ref{loccompHM}.

\section{Characterizations of Howe-Moore property}\label{sec:2}

\subsection{Ergodic theory}

The aim of this section is to prove Proposition \ref{Prop:mixing}, and generalize it to the relative situation.

Recall that a locally compact group $G$ is {\it maximally almost periodic} (resp. {\it minimally almost periodic}) if finite-dimensional unitary representations of $G$ separate points (resp. every finite-dimensional unitary representation of $G$ is trivial).

\begin{Lem}\label{finitedim} Let $N$ be a closed, non-compact normal subgroup in the locally compact group $G$, such that the pair $(G,N)$ has the relative Howe-Moore property.
\begin{itemize}
\item[a)] Let $\pi$ be a unitary representation of $G$ on the Hilbert space ${\mathcal H}$. Any finite-dimensional $\pi(N)$-invariant subspace of ${\mathcal H}$ consists of $\pi(N)$-fixed vectors. 
\item[b)] Every continuous homomorphism from $G$ to a maximally almost periodic group, factors through $G/N$.
\item[c)] $N$ is contained in the closed commutator subgroup $\overline{[G,G]}$.
\end{itemize}
\end{Lem}

\begin{proof} $(a)$ Let ${\mathcal H}^N$ be the space of $\pi(N)$-fixed vectors, and let ${\mathcal H}^\perp$ be the orthogonal subspace. As $N$ is normal in $G$, both subspaces are $\pi(G)$-invariant, and we denote by $\pi^\perp$ the restriction of $\pi$ to ${\mathcal H}^\perp$. So it is enough to prove that $\pi^\perp\vert_N$ has no finite-dimensional sub-representation. By the relative Howe-Moore property, $\pi^\perp\vert_N$ is a $C_0$-representation of $N$, while a finite-dimensional unitary representation is never $C_0$ (recall the easy argument: if $\sigma$ is finite-dimensional unitary, then $|\det\sigma(n)|=1$ for every $n\in N$; but $\det$ is a homogeneous polynomial in coefficients of $\sigma$, so not all coefficients are $C_0$). 

$(b)$ It is enough to observe that every finite-dimensional unitary representation of $G$ is trivial on $N$, which follows from $(a)$.

$(c)$ Follows from $(b)$, in view of the fact that $G/\overline{[G,G]}$ is abelian, hence maximally almost periodic.
\end{proof}

To state the ergodic theoretic characterization of the relative Howe-Moore property, we recall that if a locally compact group $G$ acts in a measure-preserving way on a probability space $(X,{\mathcal B},\mu)$, and $H$ is a closed, non-compact subgroup of $G$, the action is said to be {\it $H$-ergodic} if its restriction to $H$ is ergodic, i.e. for every $H$-invariant $A\in{\mathcal B}$, one has either $\mu(A)=0$ or $\mu(A)=1$; similarly, the action is {\it $H$-mixing} if its restriction to $H$ is mixing, i.e. $\lim_{h\rightarrow\infty}\mu(hA\cap B)=\mu(A)\mu(B)$ for every $A,B\in{\mathcal B}$. Proposition \ref{Prop:mixing} is then an immediate consequence of the more general:

\begin{Prop}\label{Prop:relmix} Let $H$ be a closed, non-compact subgroup of the locally compact group $G$. Consider the following properties:
\begin{enumerate}
\item[a)] The pair $(G,H)$ has the relative Howe-Moore property;
\item[b)] Every measure-preserving action of $G$ on a probability space which is $H$-ergodic, is necessarily $H$-mixing.
\end{enumerate}
Then $(a)\Rightarrow(b)$. The converse holds if $G$ is second countable and either $H$ is minimally almost periodic or $H$ is normal in $G$.
\end{Prop}

\begin{proof} If $G$ acts in a measure preserving way on a probability space $(X,{\mathcal B},\mu)$, by Koopman's theorem (see e.g. \cite{BeMa}, Theorem 2.1) the action is $H$-ergodic if and only if the natural $G$-representation $\pi$ on $L^2_0(X,\mu)=:\{f\in L^2(X,\mu):\int_X f\,d\mu=0\}$ has no non-zero $H$-invariant vector; and the action is $H$-mixing if and only if $\pi|_H$ is a $C_0$-representation (see \cite{BeMa}, Theorem 2.9); this proves $(a)\Rightarrow (b)$.

To prove the converse, assume that $G$ is second countable, that either $H$ is minimally almost periodic or $H\vartriangleleft G$, and that every $H$-ergodic action of $G$ is $H$-mixing. Let $\pi$ be a unitary representation of $G$ on a Hilbert space ${\mathcal H}$, without non-zero $H$-fixed vector. We must show that coefficients of $\pi|_H$ are $C_0$. Clearly we may assume that ${\mathcal H}$ is separable. Let ${\mathcal H}_{\R}$ be the real Hilbert space underlying ${\mathcal H}$, and let $\pi_{\R}$ be the representation $\pi$ viewed as an orthogonal representation on ${\mathcal H}_{\R}$. It will be enough to check that coefficients of $\pi_{\R}|_H$ are $C_0$. Since $G$ is second countable, there exists a probability space $(X,{\mathcal B},\mu)$ endowed with a measure-preserving action of $G$, such that the associated orthogonal representation $\sigma$ of $G$ on $L^2_{\R}(X,\mu)$ is equivalent to the direct sum $\bigoplus_{k=0}^{\infty}S^k\pi_{\R}$, where $S^k\pi_{\R}$ is the $k$-th symmetric tensor power of $\pi_{\R}$ (see \cite{BHV}, Corollary A.7.15).
\medskip

{\bf Claim:} The action of $G$ on $X$ is $H$-ergodic.
\medskip

The result then follows, since the assumption implies that the $G$-action on $X$ is $H$-mixing, meaning that the restriction $\tau$ of $\sigma|_H$ to the orthogonal of constants in $L^2_{\R}(X,\mu)$ is a $C_0$-representation of $H$, and since $\pi_{\R}|_H$ is a subrepresentation of $\tau$.

To prove the claim, by Koopman's theorem it is enough to check that, for every $k\geq 1$, the symmetric tensor power $S^k\pi_{\R}$ has no non-zero $H$-fixed vector. Suppose by contradiction that it does, for some $k\geq 1$. View $S^k\pi_{\R}$ as a subrepresentation of the tensor power $\pi_{\R}^{\otimes k}$. It is a standard fact (see e.g. \cite{BHV}, Proposition A.1.12) that, if the representation $\pi_{\R}^{\otimes k}$ has a non-zero $H$-fixed vector, then $\pi_{\R}|_H$ contains a finite-dimensional subrepresentation. As $H$ is either minimally almost periodic or normal in $G$, we deduce (using Lemma \ref{finitedim}(a) in the latter case) that $\pi_{\R}$ has non-zero $H$-fixed vectors, contradicting our assumption on $\pi$.
\end{proof}

\subsection{Irreducible representations and positive definite functions}

\begin{Prop}\label{irredprop} Let $H$ be a closed subgroup of the second countable, locally
compact group $G$. Assume that, for every irreducible representation
$\sigma$ of $G$ without non-zero $H$-fixed vectors, the restriction
$\sigma|_H$ is a $C_0$-representation of $H$. Then $(G,H)$ has the
Howe-Moore property.
\end{Prop}

\begin{proof} Let $\pi$ be a representation of $G$, without non-zero
$H$-fixed vector. We have to prove that $\pi|_H$ is a
$C_0$-representation of $H$. Since $G$ is second countable, by
disintegration theory (see \cite{Dix}) there exists a
$\sigma$-finite measure space $(X,\mu)$ and a measurable field $x\mapsto\sigma_x$ of irreducible representations of $G$, such
that $\pi$ is a direct integral:
$$\pi=\int_X^{\oplus}\sigma_x\,d\mu(x).$$
By assumption on $\pi$, and Proposition 2.3.2(ii) in \cite{Zim}, the
set of $x\in X$ such that $\sigma_x$ has non-zero $H$-fixed vectors,
has measure zero. So, $\mu$-almost everywhere, the restriction
$\sigma|_H$ is a $C_0$-representation of $H$, and Proposition
2.3.2(i) in \cite{Zim} applies to give the result. \end{proof}

When $H$ is normal in $G$, Proposition \ref{irredprop} can be
restated.

\begin{Cor}\label{irredcor} Let $H$ be a closed, normal subgroup of the second countable,
locally compact group $G$. The pair $(G,H)$ has the Howe-Moore
property if and only if for every \textit{irreducible}
representation $\sigma$ of $G$ which is non-trivial on $H$, the
restriction $\sigma|_H$ is a $C_0$-representation of $H$.
\end{Cor}

\begin{proof} Let $(\pi,\HH)$ be an irreducible representation of
$G$. Since $H$ is a normal subgroup of $G$, the space of
$H$-invariant vectors of $\HH$ is globally $G$-invariant. So if this
space is not zero, it has to be all $\HH$, which means that
$\pi_{|H}$ is trivial. Thus Proposition \ref{irredprop} applies.
\end{proof}

The relative Howe-Moore property can also be characterized in terms of positive definite functions.

\begin{Prop}\label{GNS} Let $H$ be a closed subgroup of the locally compact compact group $G$. The following are equivalent:
\begin{enumerate}
\item[i)] The pair $(G,H)$ has the relative Howe-Moore property;
\item[ii)] For every non-zero positive definite function $\phi$ on $G$, either there exists $c>0$ and a positive definite function $\psi$ on $H$ such that $\phi\vert_H=c+\psi$, or for every $g,g'\in G$ the function $h\mapsto \phi(ghg')$ is in $C_0(H)$.
\end{enumerate}
\end{Prop}

\begin{proof} For a positive definite function $\phi$ on $G$, let $\pi_\phi$ be the representation associated to $\phi$ by the GNS construction. The first condition in (ii) means that $\pi_\phi\vert_H$ has non-zero $H$-fixed vectors (see e.g. Proposition C.5.1 in \cite{BHV}); while the second condition in (ii) means that $\pi_\phi\vert_H$ is a $C_0$-representation (taking into account the fact that coefficients of $\pi_\phi$ are uniform limits of sums $g\rightarrow \sum_{i,j=1}^n c_i\overline{c_j}\phi(g_j^{-1}gg_i)$, see e.g. Exercise C.6.3 in \cite{BHV}). This already shows that $(i)\Rightarrow(ii)$. For the converse, let $\pi$ be a representation of $G$, without non-zero $H$-fixed vector, and let $\xi$ be a vector in the Hilbert space of $\pi$. Then $\phi(g):=\langle\pi(g)\xi|\xi\rangle$ is a positive definite function on $G$, and $\phi\vert_H$ is $C_0$ by $(ii)$.
\end{proof}

\subsection{Permanence properties}

\begin{Prop}\label{perm1} Let $L\subset H\subset G$ be closed subgroups of the locally compact group $G$.
\begin{enumerate}
\item[i)] If the pair $(H,L)$ has the relative Howe-Moore property, then so does the pair $(G,L)$.
\item[ii)] If the pair $(G,L)$ has the relative Howe-Moore property and $H$ is open in $G$, then $(H,L)$ has the relative Howe-Moore property.
\item[iii)] If the pair $(G,L)$ has the relative Howe-Moore-property, and there exists a closed normal subgroup $N\vartriangleleft G$ such that $G\simeq N\rtimes H$, then $(H,L)$ has the relative Howe-Moore-property.
\end{enumerate}
\end{Prop}

\begin{proof} (i) is clear from definitions. To prove (ii), let $H$ be an open subgroup, and let $\phi$ be a positive definite function on $H$. We check that $\phi$ satisfies condition (ii) in Proposition \ref{GNS}. This follows from the fact that $\phi$ extends to a {\it positive definite} function on $G$ (by setting it equal to 0 outside of $H$, see e.g. Exercise C.6.7 in \cite{BHV}), together with the relative Howe-Moore property for $(G,L)$. So the result follows from Proposition \ref{GNS}. Finally, (iii) follows from the fact that any representation of $H$ can be viewed as a representation of $G$ via the quotient map $N\rtimes H\rightarrow H$.
\end{proof}

\begin{Prop}\label{perm2} Let $N, L$ be closed, normal subgroups of the locally compact group $G$. If the pair $(G,N)$ has the relative Howe-Moore property, then so has the pair $(G/L,N/(N\cap L))$; moreover, the composite map $N\to G \to G/L$ is either trivial or proper. In particular, every closed normal subgroup of $G$ which is properly contained in $N$, is compact.
\end{Prop}

\begin{proof} The first statement is clear from definitions. To prove the second, we assume that $N$ is not contained in $L$, and prove that $N\cap L$ is compact. By the Gelfand-Raikov theorem (see e.g. \cite[13.6.7]{Dix}), we find a unitary irreducible representation $(\pi,{\mathcal H})$ of $G/L$ which is non-trivial on $N/(N\cap L)$.
As $N/(N\cap L)$ is normal in $G/L$, the set ${\mathcal H}^0$ of $\pi(N/(N\cap L))$-fixed vectors is a proper $\pi(G/L)$-invariant subspace of ${\mathcal H}$. By irreducibility, ${\mathcal H}^0=\{0\}$. By the relative Howe-Moore property for $(G/L,N/(N\cap L))$, coefficients of $\pi$ are $C_0$ in restriction to $N/(N\cap L)$. Since these coefficients are constant along $N\cap L$, this forces $N\cap L$ to be compact. The third statement is an immediate consequence of the second one.
\end{proof}

\section{Structural consequences of the relative Howe-Moore Property}\label{sec:3}

The aim of this section is to prove Theorem \ref{structurerelHM}, stating that, if $(G,N)$ has the relative Howe-Moore property and $N$ is normal in $G$, the structure of $N$ is quite constrained. We denote by $\T$ the multiplicative group of complex numbers of modulus 1.

\begin{Prop}\label{relHM} Let $G$ be a locally compact group, and let $N$ be a closed, non-compact normal subgroup of $G$ such that $(G,N)$ has relative Howe-Moore property. Then the action of $G$ on the set $Hom(N,\T)$ of characters of $N$, has no non-trivial fixed point.
\end{Prop}

\begin{proof} Let $\chi\in Hom(N,\T)$, and let $\pi_{\chi}=Ind_N^G\,\chi$ be the induced representation. Remember that the Hilbert space of $\pi_{\chi}$ is $${\mathcal H}_{\chi}=\{f:G\rightarrow\C, \mbox{measurable}, f(gn)=\chi(n^{-1})f(g)\, \mbox{for every} \,n\in N \,$$
$$\mbox{and almost every}\,g\in G, \int_{G/N}|f(x)|^2\,dx<\infty \}.$$
and that the representation $\pi_{\chi}$ is given by $(\pi_{\chi}(g)(f))(h)=f(g^{-1}h)$ for $g,h\in G$ and $f\in{\mathcal H}_{\chi}$.
For $g\in G$ and $n\in N$, we have
$$(\pi_{\chi}(n)f)(g)=f(n^{-1}g)=f(g.(g^{-1}n^{-1}g))=\chi(g^{-1}ng)f(g)$$
i.e.
\begin{equation}\label{indrep}
(\pi_{\chi}(n)f)(g)=(g\chi)(n)f(g).
\end{equation}

As a consequence:

\begin{equation}\label{coeffind}
\langle \pi_{\chi}(n)(f)|f\rangle = \int_{G/N}(x\chi)(n)|f(x)|^2\,dx
\end{equation}
(here $(x\chi)(n)=\chi(g^{-1}ng)$, where $x=gN\in G/N$).

{\bf Claim:} $\pi_{\chi}$ has non-zero $N$-invariant vectors if and only if $\chi\equiv 1$.

Indeed, if $\chi\equiv 1$, then $\pi_{\chi}$ is the left regular representation of $G/N$ (viewed as a representation of $G$), and it is trivial on $N$. Conversely, if $f\in{\mathcal H}_{\chi}$ is a non-zero $N$-invariant function, then by Equation (\ref{indrep}) we have $(g\chi)(n)f(g)=f(g)$ for every $n\in N$ and almost every $g\in G$. Taking $g$ in a set of non-zero measure where $f(g)\neq 0$, we get $g\chi\equiv 1$, i.e. $\chi\equiv 1$, proving the claim.
\medskip

Assume now that $(G,N)$ has the relative Howe-Moore property. Let $\chi\in Hom(N,\T)$ be a character of $N$, which is $G$-fixed. By Equation (\ref{coeffind}), we then have, for $n\in N$ and $f\in{\mathcal H}_{\chi}$:
$$\langle \pi_{\chi}(n)(f)|f\rangle =\chi(n)\|f\|^2,$$
so that the restriction of $\pi_{\chi}$ to $N$ is certainly not $C_0$. By the relative Howe-Moore property, $\pi_{\chi}$ has non-zero $N$-fixed vectors, and by the claim this implies $\chi\equiv 1$.
\hfill$\square$
\medskip

\begin{Prop}\label{open} Let $(G,H)$ be a pair with the relative Howe-Moore property. Let $U$ be an open subgroup of $G$, properly contained in $H$. Then $U$ is compact.
\end{Prop}

{\bf Proof}: Clearly we may assume that $H$ is not compact. Note that, by Proposition \ref{perm1} (ii), since $H$ is open in $G$ our assumption is equivalent to $H$ having the Howe-Moore property. First we show that $U$ has infinite index in $H$. Suppose not. Consider then the representation of $H$ on $\ell^2_0(H/U)$, the orthogonal of constants in $\ell^2(H/U)$. Since this representation has no $H$-fixed vector, it has coefficients vanishing on $H$. But these coefficients are also constant along the intersection of all conjugates of $U$, which is of finite index in $H$, this is a contradiction.

Consider then the permutation representation $\pi$ of $H$ on $\ell^2(H/U)$. Since $U$ has infinite index in $H$, it has no $H$-fixed vector, so its coefficients are $C_0$ on $H$. But the coefficient $g\mapsto\langle\pi(g)\delta_U|\delta_U\rangle$ on the characteristic function $\delta_U$ of $U$, is equal to 1 on $U$. This forces $U$ to be compact.
\end{proof}

In particular, if $G$ is a group with the Howe-Moore
property, then every proper open subgroup of $G$ is compact. If such a $G$ admits a maximal compact open subgroup $K$ (e.g. $K=\SL_n(\Z_p)$ in $G=\SL_n(\Q_p)$), then $K$ is also maximal as a subgroup.

\begin{Ex}\label{PGL} Let $\K=\R$ or $\K=\Q_p$. Set $G=\PGL_2(\K)$ and $N=\PSL_2(\K)$. The group $G$ does {\it not} have the
Howe-Moore property (because $G$ contains $N$ as a proper, open subgroup with finite index) while the pair $(G,N)$ has the
relative Howe-Moore property (since $N$ is the quotient of the Howe-Moore group $\SL_2(\K)$ by its finite center).
\end{Ex}

A group is {\it quasi-finite} if it has no proper infinite subgroup. Proposition \ref{open} says that a discrete group with the Howe-Moore
property is quasi-finite. This can be made more precise as follows.

\begin{Lem}\label{dihm}
An infinite discrete Howe-Moore group $G$ is finitely generated and finite-by-(quasi-finite simple).
\end{Lem}
\begin{proof}
If $G$ is Howe-Moore, by Proposition \ref{open}, $G$ is quasi-finite. Let $K$ be the locally finite (= locally elliptic) radical of $G$, i.e. the subgroup generated by all finite normal subgroups of $G$. If $K$ is finite, then it is clear that $G/K$ is finitely generated and still Howe-Moore, hence simple. Otherwise, $G$ is locally finite and by a classical result of Hall and Kulatilaka \cite{HK}, $G$ is isomorphic to $C_{p^\infty}=\Z[1/p]/\Z$ for some prime $p$. But an infinite abelian group is not Howe-Moore, e.g. by Lemma \ref{finitedim}(c).
\end{proof}

\begin{Rem}\label{quasifini}
Quasi-finite finitely generated groups are known to exist
\cite{Ols}, but essentially nothing is known about their unitary
representations \footnote{However, since every non-elementary
torsion-free word hyperbolic group has quasi-finite quotients
\cite{Ols2}, there exist quasi-finite groups with Property (T).}.
\end{Rem}

\begin{Prop}
Let $G$ be a locally compact group, $N$ an infinite normal discrete subgroup such that $(G,N)$ has the relative Howe-Moore Property. Suppose that the centralizer of $N$ is open in $G$ (for instance, this holds if $N$ is finitely generated, or if $G$ is a Lie group). Then $N$ is finitely generated, and finite-by-(quasi finite simple) (in particular it is torsion), and $N/K$ has the Howe-Moore Property for some normal finite subgroup $K$.\label{dino}
\end{Prop}
\begin{proof}
Let us first check the ``for instance" assertion. In case $G$ is a Lie group, $G_0$ is open and since $N$ is discrete, it is centralized by $G_0$. Also, if $N$ is finitely generated, its automorphism group is discrete and the centralizer of $N$ is open.

Let us prove the main assertion. Let $C$ be the centralizer of $N$, which is open and normal in $G$. Then $CN$ is open in $G$, so $(CN,N)$ has the relative Howe-Moore Property by Proposition \ref{perm1}(ii). If $N$ is abelian, it is central in $CN$, so by Proposition \ref{relHM}, $N$ is trivial, a contradiction. Therefore by Proposition \ref{perm2}, $K=N\cap C$ is finite. Then by Proposition \ref{perm2}, $(NC/C,N/K)$ has the Howe-Moore Property. Since $NC/C=N/K$, we obtain that $N/K$ has the Howe-Moore Property. The remaining then follows from Lemma \ref{dihm}.
\end{proof}

Recall that a locally compact group $G$ is {\it characteristically simple} if the only topologically characteristic subgroups of $G$ are the trivial subgroups. The following result contains Proposition \ref{loccompHM} (obtained by taking $G=N$). Denote by $W(G)$ the union of all normal compact subgroups of $G$. This is a characteristic subgroup of $G$, not always closed (see Example \ref{WuYu}). We will need a result of Platonov \cite{Pla}: $G$ admits a closed locally elliptic subgroup containing all closed normal locally elliptic subgroups: this subgroup is called the {\it elliptic radical} of $G$, denoted by $R_{ell}(G)$. Clearly $W(G)\subset R_{ell}(G)$.

\begin{Thm}\label{structurerelHM} Let $G$ be a locally compact group, and let $N$ be a closed, non-compact normal subgroup of $G$ such that $(G,N)$ has relative Howe-Moore property.
Then

\begin{itemize}
\item[(1)] either $N$ is locally elliptic, not compactly generated, and $W(N)$ is dense in $N$,
\item[(2)] or $W(N)$ is compact and $N/W(N)$ is characteristically simple.
If moreover $N$ is compactly generated, then one of the following cases occurs:
\begin{itemize}
\item[(a)] $N/W(N)$ is isomorphic to $\R^n$ for some $n\geq 1$, and the $G$-representation on $N/W(N)$ is irreducible and non-trivial.
\item[(b)] $N/W(N)$ is a topologically simple group; if in addition it is discrete, then it is Howe-Moore;
\end{itemize}
\end{itemize}
\end{Thm}

\begin{proof} By Proposition \ref{perm2}, $\overline{W(N)}$ is either equal to $N$, or compact. Clearly, if it is compact, then by definition $\overline{W(N)}\subset W(N)$; so $N/W(N)$ has no non-trivial compact characteristic subgroup. By the Howe-Moore Property and Proposition \ref{perm2}, it is characteristically simple.

Otherwise, $W(N)$ is dense in $N$, so that $N=R_{ell}(N)$, i.e. $N$ is locally elliptic. By \cite[Lemma~1]{Pla}, in a locally elliptic locally compact group, every compact subset is contained in a compact subgroup; therefore as $N$ is not compact, it cannot be compactly generated.

Let us now suppose that $N$ is compactly generated. In particular, $W(N)$ is compact and $N/W(N)$ is characteristically simple. Since $W(N)$ does not play any role in the sequel, we may assume $W(N)=1$.

Before proceeding, let us say that a topological group $H$ is a {\it quasi-product} if there exists pairwise commuting closed normal subgroups $H_1,H_2,...,H_k$, with $H_i\cap H_j=\{1\}$ for $i\neq j$, such that $H=\overline{H_1...H_k}$. By a result by Caprace and Monod (\cite{CaMo}, Corollary D): a compactly generated, non-compact, characteristically simple locally compact group is either a vector group, or discrete, or the quasi-product of its minimal closed normal subgroups (which are finitely many pairwise isomorphic nonabelian topologically simple groups). Let us consider the three cases successively.
\begin{itemize}
\item First case: $N\simeq \R^n$ is a vector group. The action of $G$ is irreducible and non-trivial by Propositions \ref{perm2} and \ref{relHM}.

\item Second case: $N$ is discrete. Then (b) follows from Proposition \ref{dino}.

\item Third case: $N$ is a quasi-product of its minimal closed normal subgroups $S_1,...,S_k$, which are nonabelian noncompact simple groups permuted by $G$. The homomorphism $\alpha: G\rightarrow \textnormal{Sym}(k)$ given by the $G$-action on the $S_i$'s, is continuous. Indeed, for all $i=1,...,k$, define $A_i=\{g:[gS_ig^{-1},S_i]=1\}$. Then $A_i$ is closed. So the union of all $A_i$'s is closed, but this is just the complement of the kernel $G_1$ of $\alpha$, which is therefore open. By Proposition \ref{perm1}(ii), $(G_1,N)$ has the Howe-Moore Property. Since $S_1$ is normal in $G_1$, we deduce from Proposition \ref{perm2} that $N=S_1$. We do not know if, in this case, $N$ is itself a Howe-Moore group.\qedhere
\end{itemize}\end{proof}

\begin{Ex} An example of a pair $(G,N)$ with the relative Howe-Moore property and $N$ locally elliptic, not compactly generated, is $(\GL_n(\Z_p)\ltimes\Q_p^n,\Q_p^n)$, for $n\ge 1$, by Theorem \ref{main}.
Note that this group is entirely elliptic. Also, using the isomorphism between $\Q_p/\Z_p$ and the discrete group $\Z[1/p]/\Z$, we see that $(\GL_n(\Z_p)\ltimes(\Z[1/p]/\Z)^n,(\Z[1/p]/\Z)^n)$ has the relative Howe-Moore Property (note that in this case, Proposition \ref{dino} does not apply).
\end{Ex}

\begin{Ex}\label{WuYu} The following example is due to Wu-Yu \cite{WY}. Let $(p_i)$ be any infinite family of odd primes (which can be constant, or injective). Consider the semidirect product
$$\left(\bigoplus_i\Z/p_i\Z\right)\rtimes\left(\prod_i(\Z/p_i\Z)^*\right),$$
with the multiplication action on each factor. Then each $\Z/p_i\Z\rtimes(\Z/p_i\Z)^*$ is a compact normal subgroup, and they generate a dense subgroup.
 However, any element $(0,\sigma)$ with $\sigma$ of infinite support has an unbounded conjugacy class, and is therefore not contained in a compact normal subgroup.
\end{Ex}

\section{Semidirect products}\label{sec:4}

The main result of this section is Theorem \ref{HMforsemidir}, providing non-trivial examples of pairs with the relative Howe-Moore property. We need the following general elementary lemma, that the reader can prove as an exercise.

\begin{Lem}\label{GHV} Let $G$ be a group, let $V$ be a vector space over any field, and let $\pi$ be a representation of $G$ on $V$. The following are equivalent:
\begin{enumerate}
\item[i)] $\pi$ is irreducible and non-trivial;
\item[ii)] the only $\pi(G)$-invariant {\it affine} subspaces of $V$ are $\{0\}$ and $V$;
\item[iii)] for every non-zero $v\in V$, the orbit $\pi(G)v$ is not contained in any affine hyperplane.
\end{enumerate}
\end{Lem}

To prove Theorem \ref{main}, we make use of the Mackey machinery
describing irreducible representations of semidirect products with
an abelian normal subgroup, and some results about vanishing at
infinity of oscillatory integrals.

\begin{Thm}\label{HMforsemidir} Let $\K$ be a local field of characteristic 0, let $V$ be a finite-dimensional $\K$-vector space, and let $G$ be a second countable locally compact with a finite-dimensional representation $G\to\GL(V)$. If the $G$-orbit of every non-zero vector in the dual group $\widehat{V}$, is locally closed and not locally contained in some affine hyperplane, then the pair $(G\ltimes V,V)$ has the relative Howe-Moore property.
\end{Thm}

{\bf Proof:} Let $\pi$ be a unitary representation of $G\ltimes V$, without non-zero $V$-invariant vectors. We must prove that $\pi|_V$ is a $C_0$-representation. Thanks to Corollary \ref{irredcor}, we may assume that $\pi$ is irreducible.

Because our group $G\ltimes V$ is a semidirect product, with $V$ a finite-dimensional $\K$-vector space, we may appeal to the Mackey machine for
representations of semidirect products.

Consider the action of $G$ on $\widehat{V}$ given by $(g.\chi)(v)=\chi(g^{-1}.v)$ for all $g\in G$, $v\in V$ and $\chi
\in\widehat{V}$. We summarize in the following proposition some relevant facts from Mackey theory for semidirect products (see section 2.2 in \cite{Mackey}).

\begin{Prop}\label{mackey}
Let $(\HH,\rho)$ be a  unitary representation of $G\ltimes V$.

\begin{enumerate}
    \item There exists on $\widehat{V}$
    a projection-valued regular Borel measure
    $E: \mathrm{Borel}(\widehat{V}) \to \HH$ such that
    \begin{enumerate}
        \item $\rho(v) = \int_{\widehat{V}} \chi(v) dE(\chi)$, for all $v\in V$;
        \item $E(g.B)=\rho(g) E(B) \rho(g^{-1})$ for all $ g\in G$ and $B\in\mathrm{Borel}(\widehat{V})$.
    \end{enumerate}

    \item For $\xi\in\HH$ set $\mu_\xi = \left\langle  E(.)\xi, \xi\right\rangle $; this is a regular positive Borel measure
    such that $\left\langle \rho (v) \xi, \xi\right\rangle
        = \int_{\widehat{V}} \chi(v) d\mu_{\xi}(\chi)$, for all $v\in V$;

    \item There exists a regular Borel measure $\mu$ on
    $\widehat{V}$  satisfying the following statements:

    \begin{enumerate}
        \item for any $\xi\in \HH$, the measure $\mu_\xi$
        is absolutely continuous with respect to $\mu$.
        There exists an $L^2$-function $\widehat{\xi}: \widehat{V}\to \HH $
        such that
        $\Vert \widehat{\xi}(\chi)\Vert^2 =\dfrac{d\mu_\xi}{d\mu}(\chi)$
         is the Radon-Nikodym derivative of $\mu_\xi$
        with respect to $\mu$ and
        \begin{enumerate}
            \item  [i)] $\xi = \int_{\widehat{V}} \widehat{\xi}(\chi) d\mu(\chi)$;

            \item [ii)] $\left\langle \rho (v) \xi, \xi\right\rangle
            = \int_{\widehat{V}} \chi(v)\Vert \widehat{\xi}(\chi)\Vert^2 d\mu(\chi)$,
            for all $v\in V$;
        \end{enumerate}
        \item the measure $\mu$ is quasi-$G$-invariant; that is,
        for every measurable $A\subset \widehat{V}$,  $\mu(g.A) = 0$ if and only if
        $\mu(A) = 0$.
        \item If moreover the representation $\pi$ is irreducible,
        then the measure $\mu$ is $G$-ergodic.
        That is,  for any $G$-invariant
        measurable set $A$, we have either $\mu (A)=0$, or
        $\mu (\widehat{V}\setminus A)=0$ .
        \hfill$\square$
\end{enumerate}
\end{enumerate}
\end{Prop}

Let $(\pi,\HH)$ be an irreducible unitary representation of
 $G\ltimes V$, without non-zero $V$-invariant vector. Denote by $\mu$ the quasi-$G$-invariant
 $G$-ergodic measure given by Proposition \ref{mackey}.
 By Proposition \ref{mackey} (3.a(ii)), it is enough to prove that,
 for any positive function $f$ in $L^{1}(\mu)$,
\begin{equation}\label{mixing}
\lim_{| v| \to\infty} \int_{\widehat{V}} \chi(v)f(\chi)\,
\mathrm{d}\mu(\chi) = 0;
\end{equation}
equivalently, we must establish the decay at infinity of the Fourier transform of the possibly singular measure $f\,d\mu$ (here ``singular'' means with respect to Lebesgue measure on $\widehat{V}$).

Since the $G$-orbits on $\widehat{V}$ are assumed to be locally closed , any $G$-ergodic measure $\mu$ on
$\widehat{V}$ is concentrated on a single $G$-orbit
$\mathcal{O}=G.\chi$ for some $\chi\in \widehat{V}$ (\cite{Zim},
Proposition 2.1.10); note that $\chi\neq 0$ as $\pi$ has no $V$-invariant vector. Since any two quasi-invariant measures on
$\mathcal{O}$ are equivalent (see e.g. \cite{BHV}, Theorem B.1.7), to prove (\ref{mixing}) we may replace the quasi-invariant measure $\mu$ by Lebesgue measure $\sigma$ on $\mathcal{O}$.

Denote by $v\cdot w=\sum_{i=1}^n v_iw_i$ the standard scalar product on $\K^n$. Let $M$ be a smooth $d$-dimensional sub-manifold of $\K^n$, locally given by a parametrization $\phi:U\rightarrow \K^n$, where $U$ is a neighborhood of $0$ in $\K^d$. Fix a point $x_0\in U$; we say that $M$ has finite type at $\phi(x_0)\in M$ if, for every non-zero vector $\eta\in\K^n$, the function $x\mapsto(\phi(x)-\phi(x_0))\cdot\eta$ does not vanish to infinite order at $x=x_0$. The {\it type} of $M$ at $x_0$ is then the smallest $k\geq 1$ such that, for every non-zero vector $\eta\in\K^n$, there exists a multi-index $\alpha$ with $1\leq|\alpha|\leq k$, such that $\partial^\alpha_x(\phi(x)\cdot\eta)\neq 0$ at $x=x_0$. Say that $M$ has {\it finite type} if the supremum $k$ of the types at each point of $M$, is finite. If $M$ has finite type $k$, then for any $C^\infty$-function $g$ with compact support in $M$, we have
\begin{equation}\label{TF}
\lim_{v\in\K^n,|v|\rightarrow \infty}\int_{M}\lambda_0(v\cdot x)g(x)d\sigma(x)=O(|v|^{-1/k})
\end{equation}
(where $\lambda_0$ is a non-trivial character on $\K$; if $\K$ is non-archimedean, assume that $\lambda_0$ is non-trivial on the valuation ring $R$ of $\K$, but trivial on the maximal ideal of $R$): for a proof of (\ref{TF}), see \cite{Stein}, Theorem 2 in Chapter VIII in the case where $\K$ is archimedean, and \cite{Clu}, Theorem 3.11 in the case where $\K$ is non-archimedean of characteristic 0.

Assume we know that our orbit $\mathcal{O}$ is a sub-manifold of finite type in $\widehat{V}$. Since every character of $\widehat{V}$ can be written as $\chi\mapsto \lambda_0(v\cdot\chi)$ for some $v\in V$ (see e.g. Theorem 3 in Chapter II in \cite{Wei}), we see that (\ref{TF}) implies (\ref{mixing}), using density of $C^\infty$-functions with compact support in $L^1$-functions.

So it remains to show that $\mathcal{O}$ has finite type. By homogeneity, it is enough to prove that it has finite type at every point. But since $\mathcal{O}$ is a $\K$-analytic sub-manifold, having infinite type at $\chi\in\mathcal{O}$ would mean that $\mathcal{O}$ is locally contained in some affine hyperplane $H$, contrary to our assumptions. This completes the proof.
\hfill$\square$
\medskip

\begin{Prop}\label{orblocc}
Let $G$ be a connected Lie group acting irreducibly non-trivially on a finite-dimensional real vector space $V$. Then the orbits of $G$ on $V$ are locally closed and not locally contained in any affine hyperplane.
\end{Prop}
\begin{proof}
Let us first prove that the orbits are not locally contain in any hyperplane. We identify $G$ with its image into $\GL(V)$, not assuming that it is closed. If a nonzero orbit is locally contained in an affine hyperplane $A$, at some vector $v\in V$, as the orbit map $G\rightarrow V:g\mapsto gv$ is real analytic, we see that the orbit $Gv$ is contained in $A$. By Lemma \ref{GHV}, this forces $v=0$.


Let us now prove the first assertion.
We can suppose that $G$ acts faithfully. Since $G$ acts irreducibly on $V$, it is reductive. Write $G=SZ$, with $S$ semisimple and $Z$ central, both connected. Let $\K$ be the $\R$-subalgebra of $\textnormal{End}(V)$ generated by $Z$. Since the action is irreducible, $\K$ is a field, so is either $\R$ or isomorphic to $\C$. In case $\K=\R$, the group $G$ maps with finite index into its Zariski closure in $\GL(V)$, so its orbits are locally closed \cite[Theorem~3.1.3]{Zim}. So we can suppose $\K=\C$, and thus view $V$ as a $\C$-vector space, $Z$ acting by scalar multiplication.

{\bf Claim:} Fix $x\in V$. Let $L$ be the global stabilizer of $\C x$ in $S$. (So the global stabilizer of $\C x$ in $G$ is $LZ$). Then we claim that the orbits of $LZ$ in $\C^*x$ are closed. This is because $L$ is real-Zariski-closed in $S$, so it has finitely many components, so $LZ$ as well. Now any Lie subgroup of $\C^*$ being closed, the image of the mapping of $LZ$ into $\C^*$ is closed.

Let us now prove the desired
assertion. We know that $\C^*S$ has finite index in its real Zariski closure, so its orbits on $V$ are locally closed \cite[Theorem~3.1.3]{Zim}. Take $x\in V$; we can suppose $x\neq 0$. Then there exists a neighbourhood $\Omega$ of $x$ such that $\C^*Sx\cap\Omega$ is closed in $\Omega$. Let us show that $ZSx\cap\Omega$ is closed in $\Omega$. Pick a sequence $(z_ns_nx)$ converging to $y\in\Omega$, with $z_n\in Z$, $s_n\in S$. So there exists $\lambda\in\C^*$ and $s\in S$ such that $y=\lambda s x$. So setting $\sigma_n=s^{-1}s_n \in S$, we have
$$\lim_{n\to\infty}\sigma_n(z_n\lambda^{-1})x=x.$$
In particular, in the complex projective space $(\sigma_n[x])$ tends to $[x]$ (denoting by $[x]$ the class of $x$ modulo $\C^*$). Since the stabilizer of $[x]$ is $L$, and since the orbits of $S$ in the projective space are locally closed, we can write $\sigma_n=\eps_n\ell_n$, where $\eps_n\in S$, $\eps_n\to 1$, and $\ell_n\in L$. By continuity of the action, we obtain $(z_n\lambda^{-1})\ell_n x\to x$, that is
$$\lim_{n\to\infty} z_n\ell_n x = \lambda x.$$
By the Claim, $ZLx$ is closed in $\C^*x$. Therefore $\lambda x\in ZLx$.
\end{proof}

\begin{Thm} Let $G$ be a Lie group and $V$ a vector group which is closed and normal in $G$. Let $G_0$ be its connected component of identity. The following are equivalent:
\begin{enumerate}
\item[(i)] The pair $(G,V)$ has the relative Howe-Moore property;
\item[(ii)] The pair $(G_0,V)$ has the relative Howe-Moore property;
\item[(iii)] $G_0$ acts irreducibly and non-trivially on $V$.
\end{enumerate}\label{semilie}\end{Thm}
\begin{proof}~

\begin{itemize}\item $(i)\Rightarrow (ii)$. Since $G_0$ is open in $G$ for the Hausdorff topology, Proposition \ref{perm1} (ii) applies.

\item $(ii)\Rightarrow(iii)$ Follows from Propositions \ref{perm2} and \ref{relHM}.

\item $(iii)\Rightarrow (i)$. Let us show that $(G\ltimes V,V\times V)$ has the relative Howe-Moore Property, the result following then from Proposition \ref{perm2}.

Note that the $G_0$-representation on $\widehat{V}$ is irreducible and non-trivial. Therefore, by Proposition \ref{orblocc}, $G$-orbits on $\widehat{V}$ are locally closed and, except $\{0\}$, not locally contained in any affine hyperplane.
The result then follows from Theorem \ref{HMforsemidir}.\qedhere
\end{itemize}
\end{proof}

\begin{Thm} Let $G$ be a locally compact totally disconnected group with a continuous representation in $\GL(V)$, where $V$ is a finite-dimensional vector space over $\Q_p$.
The following are equivalent:
\begin{enumerate}
\item[(i)] The pair $(G\ltimes V,V)$ has the relative Howe-Moore property;
\item[(ii)] every open subgroup of $G$ acts irreducibly and non-trivially on $V$.
\end{enumerate}\label{semipad}\end{Thm}
\begin{proof}
Suppose (i). Then it follows from Propositions \ref{perm1}(ii) and \ref{relHM} that the action of $G$ on $V$ is irreducible and non-trivial in restriction to any open subgroup.

Conversely suppose (ii). Let $K$ be an open compact subgroup in $G$. It is enough to check that $(K\ltimes V,V)$ has the relative Howe-Moore Property. Since the orbits of $K$ in $\widehat{V}$ are obviously closed, we just need, to apply Theorem \ref{HMforsemidir}, to check that orbits are not locally contained in any affine hyperplane. If this were the case, upon replacing $K$ by some finite index subgroup, we would obtain an orbit entirely contained in an affine hyperplane. By Lemma \ref{GHV}, this is not compatible with the fact that the action of $K$ is irreducible and non-trivial.
\end{proof}

\section{The relative Howe-Moore property for Lie groups}\label{sec:5}

The purpose of this section is to prove Theorem \ref{structureLIEHM}, which is the archimedean part of Theorem \ref{main}.
By a {\it Lie group} we mean a real Lie group, without any connectedness assumption.

\begin{Prop}\label{LIEcar} Let $G$ be a non-compact Lie group. Suppose that every proper characteristic subgroup of $G$ is compact. Then there exists a characteristic compact normal subgroup $K$ such that one of the following cases occurs:
\begin{itemize}
\item[(a)] $K$ is a central torus and $G/K$ is isomorphic to $\R^n$ for some $n\geq 1$;
\item[(b)] $G$ is connected reductive with dense isotypic non-compact Levi factor; $K$ is a central subgroup.
\item[(c)] $K=G_0$. 
\end{itemize}
\end{Prop}
\begin{proof}
Since $G_0$ is closed characteristic, either $G_0$ is compact, or $G$ is connected, as we now suppose. Let $R(G)$ be the solvable radical of $G$. If $R(G)$ is compact, then it is a torus, central by connectedness of $G$. So $G$ is reductive, with a Levi factor $S$ (maybe not closed) so that $G$ is locally isomorphic to $S\times R(G)$. Then $S$ is not compact, so it is dense, and $S$ clearly has to be isotypic.

Otherwise $G$ is solvable; this implies that $\overline{[G,G]}$ is compact, so this is a torus. Therefore if $K$ denotes the maximal torus in $G$, $K$ is normal and $G/K$ is a vector group.
\end{proof}

\begin{Thm}\label{structureLIEHM} Let $G$ be a Lie group having a faithful continuous real finite-dimensional representation, and let $N$ be a closed, non-compact normal subgroup of $G$ such that $(G,N)$ has relative Howe-Moore property. Then $(G,N)$ has the relative Howe-Moore Property if and only if one of the following cases occurs:
\begin{itemize}
\item[(a)] $N$ is isomorphic to $\R^n$ for some $n\ge 1$, and the representation of $G$ on $N$ is irreducible and non-trivial;
\item[(b)] $N$ is a connected non-compact simple (linear) Lie group.
\end{itemize}
\end{Thm}
\begin{proof}
Suppose that $(G,N)$ has the relative Howe-Moore Property. By Proposition \ref{perm2}, every characteristic closed subgroup of $N$ is compact, so we can apply Proposition \ref{LIEcar}; we confront the three cases provided by that proposition to the conclusions of Theorem \ref{structurerelHM}.

\begin{itemize}
\item (a) $N$ has a compact central torus $K$ such that $N/K\simeq\R^n$. Then $N$ is connected. By Theorem \ref{structurerelHM} (case 2.a), $K=W(N)$ and $G$ acts irreducibly and non-trivially on $N/K$.

We still have to prove that $K=\{1\}$. By Proposition \ref{perm1}(ii), $(G_0,N)$ has the relative Howe-Moore Property as well, so we can suppose that $G$ is connected. Therefore $K$ is central in $G$. Let us consider $G$ as endowed with a continuous faithful {\it complex} representation $V$. Since $K$ is compact and central, it acts diagonally on $V$; so writing $V$ as a sum of $K$-isotypic subspaces $V_i$, the group $G$ preserves each $V_i$, on which $K$ acts by scalars. Fix one $i$. Consider the determinant map $G\to\C^*$ for the action of $G$ on $V_i$. Since $N$ is contained in $\overline{[G,G]}$ (by Lemma \ref{finitedim}(c)), the determinant map on $V_i$ is trivial on $N$, and in particular on $K$, which acts by scalars. This shows that $K$ acts by $d_i$-th roots of unity, with $d_i$ the complex dimension of $V_i$; since $K$ is connected, this shows that $K$ acts trivially on $V_i$. Since this holds for any $i$, this shows that $K$ acts trivially on $V$ and therefore $K=\{1\}$ by faithfulness.

\item (b) $N$ is connected reductive with dense Levi factor $S$. Since $G$ is linear, $S$ is closed (see e.g. Theorem 4.5 in \cite{Hoch}) and therefore $S=N$ and $W(N)=\{1\}$. By Theorem \ref{structurerelHM} (case 2.b), $S$ is a simple, non-compact, linear Lie group.

\item (c) Suppose that $N_0$ is compact. Let us show that this case cannot occur. If $N\subset G_0$, then $N/N_0$ is normal and discrete in $G_0/N_0$, hence central. In particular, it is finitely generated, infinite (because $N$ is not compact) and thus has a finite index proper characteristic subgroup, a contradiction. So $N\cap G_0$ is properly contained in $N$, so is compact. By Proposition \ref{perm2}, $(G/G_0,N/(N\cap G_0))$ has the relative Howe-Moore Property, so by Proposition \ref{dino} the infinite discrete group $N/(N\cap G_0)$ has the Howe-Moore Property. Identify $G$ with its image in $\GL(V)$. Let $A$ be the normalizer in $\GL(V)$ of $N\cap G_0$; it contains $N$ and is Zariski closed, because $N\cap G_0$ is compact, hence Zariski closed. So $A/(N\cap G_0)$ is $\R$-linear, hence the infinite discrete Howe-Moore group $N/(N\cap G_0)$ is $\R$-linear which contradicts Lemma \ref{dihm}.
\end{itemize}

Conversely, the pairs given in the theorem have the relative Howe-Moore Property. In Case (b), this is part of the main result in \cite{HowMoo}; in Case (a), first $(G\ltimes N,N)$ has the relative Howe-Moore Property by Theorem \ref{semilie}. Then by Proposition \ref{perm2}, $(G,N)$ has the relative Howe-Moore Property.
\end{proof}

\section{The relative Howe-Moore property for analytic $p$-adic groups}\label{sec:6}
In this section, our goal is to prove Theorem \ref{structureP-ADIC}, which is the non-archimedean part of Theorem \ref{main}. In all this part, $p$ is a fixed prime number.

\begin{Lem}\label{normalsolv} Let $G$ be an analytic $p$-adic group (or more generally, any closed subgroup of $\GL(n,\K)$ for $\K$ a Hausdorff topological field). If $G$ possesses an open solvable subgroup, then $G$ also possesses a normal open solvable subgroup.
\end{Lem}
\begin{proof}
Since any decreasing sequence of Zariski-closed subsets stabilizes, there exists an open (for the given topology) subgroup $U$ of $G$ for which the Zariski closure $V=\overline{U}^Z$ of $U$ in $G$ is minimal, i.e.~does not properly contain the Zariski closure of any other open subgroup of $G$. It follows that for any open subgroup $U_1$ contained in $U$, $\overline{U_1}^Z=V$. Besides, if $gUg^{-1}$ is any conjugate of $U$, then it satisfies the same property as $U$ and therefore $\overline{U_2}^Z=\overline{gUg^{-1}}^Z=gVg^{-1}$ for every open subgroup $U_2$ contained in $gUg^{-1}$. Applying this to $U_2=U\cap gUg^{-1}$, we obtain $gVg^{-1}=V$, so $V$ is normal. Moreover, taking $U_2$ to be solvable, we see that $V$ is solvable as well.
\end{proof}

\begin{Lem}\label{padicLie} Let $A$ be a non-compact, abelian $p$-adic Lie group, such that every closed characteristic subgroup is compact. Then either $A$ is isomorphic to $\Q_p^n$ (for some $n\geq 1$), or $A$ is discrete of one of the following form
\begin{itemize}
\item an arbitrary vector space over $\Q$;
\item an arbitrary vector space over $\Z/\ell\Z$ for some prime $\ell$;
\item an artinian divisible group $(\Z[1/\ell]/\Z)^k$ for some prime $\ell$ and integer $k\ge 1$.
\end{itemize}
Conversely, all proper characteristic subgroups of these groups are compact.
\end{Lem}

\begin{proof}The last assertion is easy and left to the reader.

Assume that every closed proper characteristic subgroup of $A$ is compact. The elliptic radical $R_{ell}(A)$ is a closed characteristic subgroup, so we separate two cases. Note that $R_{ell}(A)$ is open, as $A$ admits compact open subgroups.

\begin{enumerate}
\item[a)] $R_{ell}(A)$ is compact. As a compact abelian $p$-adic Lie group, it is isomorphic to the direct product of a finite group and $\Z_p^k$ for some $k$; in particular, it has no divisible element. Hence $A/R_{ell}(A)$ is a torsion-free discrete abelian group, which is characteristically simple. Since for any $m>0$ the subgroup of $m$-divisible points is characteristic, the group $A/R_{ell}(A)$ is divisible; actually, being torsion-free it is a non-zero $\Q$-vector space. We claim that $A$ is the direct product of $R_{ell}(A)$ and $A/R_{ell}(A)$. Indeed, the dual group of $A/R_{ell}(A)$, which is a compact connected group, is a closed subgroup of $\widehat{A}$. Since a compact connected abelian group is divisible, it has a direct summand in the ambient group (forgetting the topology); since it is open, the direct factor is discrete, hence closed. So $A/R_{ell}(A)$ is (topologically) a direct summand in $A$. As $A/R_{ell}(A)$ is the set of divisible points in $A$, it is a characteristic subgroup. The assumption then forces $R_{ell}(A)$ to be trivial, and $A$ is a $\Q$-vector space with the discrete topology.

\item[b)] $R_{ell}(A)=A$. For some $n\geq 0$, the Lie algebra of $A$ is isomorphic to the abelian Lie algebra
$\Q_p^n$. By \cite{Bou} (III.7.1, Th\'eor\`eme 1), $A$ has an open subgroup $M$ isomorphic to $\Z_p^n$. Let $T$ be the torsion subgroup of $A$. As $M$ is torsion-free, $T$ is discrete in $A$. So $T$ is a closed characteristic subgroup. Again, we have two cases. If $T=A$, then $A$ is discrete and torsion. It is therefore the direct sum of its $\ell$-components ($\ell$ ranging over primes), so the assumption on characteristic subgroups implies that $A$ has $\ell$-torsion for only one prime $\ell$. If $\ell A=A$, then $A$ is divisible, hence of the form $(\Z[1/\ell]/\Z)^{(I)}$. Its $\ell$-torsion is a proper subgroup, so is finite, so $I$ is finite. Otherwise, $\ell A$ is a proper subgroup of $A$, so is finite, so $\ell^kA=\{0\}$ for some $k$. If the $\ell$-torsion subgroup of $A$ were finite, then by induction, so would be the $\ell^n$-torsion, so $A$ would be finite, a contradiction. So the $\ell$-torsion is all of $A$.

If $T\neq A$, then $T$ is compact, hence finite.  Let us show that $B=:A/T$ is isomorphic to $\Q_p^n$. The subgroup $pB$ contains $pM$ and therefore is open, hence closed. It is also a characteristic subgroup, so is compact or $pB=B$. Assume by contradiction that $pB$ is compact. We claim that $pB$ has finite index in $B$, at most the index $n$ of $p^2B$ in $pB$. Indeed, if $x_1,\dots,x_{n+1}\in B$, then at least two of $px_1,\dots,px_{n+1}$ coincide in $pB/p^{2}B$, say $px_i-px_j=p^2y$. Since $B$ is torsion-free, this implies $x_i-x_j=py$ and shows that $pB$ has finite index in $B$, so $B$ is compact, a contradiction. So $pB=B$. The homomorphism $u_p:B\to B:x\mapsto px$ is therefore bijective, and both this map and its inverse are continuous on the neighborhood $pM$ of zero in $A$, so $u_p$ is
a bi-continuous automorphism. So the union $H=\bigcup_{k>0}u_p^{-k}(M)$
is the direct limit of the sequence of homomorphisms $u_p^{-1}:\Z_p^n\to p^{-1}\Z_p^n$,
so $H$ is isomorphic to $\Q_p^n$. We see that $H$ does not depend on the
choice of $M$ (since any other choice $M'$ would contain $p^kM$ for some
$k$), so $H$ is a characteristic subgroup, hence $H=B$. Again using Pontryagin duality, and the fact that $\Q_p^n$ has divisible dual (isomorphic to itself), we see that the group $A$ is the direct product of $T$ and $\Q_p^n$, with $n>0$. The factor $\Q_p^n$ is characteristic (as the set of divisible points in $A$), so $T$ is trivial.\qedhere
\end{enumerate}
\end{proof}

For any group $G$, we define the {\it radical} $R(G)$ of $G$ as the subgroup generated by closed, normal, solvable subgroups of $G$. This is a characteristic subgroup of $G$. If $G$ is a closed subgroup of $\GL_n(K)$, where $K$ is a topological Hausdorff field, then $R(G)$ is solvable and closed in $G$.

By a $p$-adic group of simple type we mean $S^+$, with $S=\mathcal{S}(\Q_p)$, where $\mathcal{S}$ is a simple isotropic $\Q_p$-algebraic group and $S^+$ is the subgroup of $S$ generated by unipotent elements.

\begin{Lem}
Let $G$ be a $p$-adic Lie group. Suppose that any proper closed characteristic subgroup of $G$ has infinite index. Let $K$ be a closed, compact normal subgroup of $G$. Then $K$ is central in $G$.\label{comcen}
\end{Lem}
\begin{proof}
Since $K$ is a compact $p$-adic Lie group, it is topologically finitely generated. Therefore it has a decreasing sequence of (open) characteristic finite index subgroups $(K_n)$ such that $\bigcap_nK_n=\{1\}$. The action of $G$ on $K/K_n$ has closed kernel of finite index; this is a characteristic subgroup of $G$, so this means that this is all of $G$, i.e.~the action of $G$ on $K/K_n$ is trivial for all $n$. This implies that $K$ is central in $G$.
\end{proof}

\begin{Lem}
Let $G$ be an analytic $p$-adic Lie group, $\mathfrak{g}$ its Lie algebra, $\mathfrak{r}$ the radical of $\mathfrak{g}$. Then the radical $R(G)$ of $G$ has $\mathfrak{r}$ as Lie algebra.\label{radlie}
\end{Lem}
\begin{proof}
First note that $\mathfrak{r}$ is the Lie algebra of a compact solvable subgroup $R_1$ of $G$. The adjoint action of $G$ on $\mathfrak{g}/\mathfrak{r}$ is trivial on $R_1$ (provided $R_1$ is small enough); let $R_2$ be the kernel of this action. Then $R_2$ contains $R_1$, and has an open solvable subgroup. By Lemma \ref{normalsolv}, the radical $R$ of $R_2$ is open in $R_2$; it contains a finite index subgroup of $R_1$, so its Lie algebra contains $\mathfrak{r}$.
Since $R$ is characteristic in $R_2$ which is normal in $G$, $R$ is normal in $G$. Therefore $R(G)$ contains $R$, so its Lie algebra contains $\mathfrak{r}$; conversely the Lie algebra of $R(G)$ is a solvable ideal of $\mathfrak{g}$, so it is contained in $\mathfrak{r}$.
\end{proof}

The following result provides in particular a description of characteristically simple analytic $p$-adic groups.

\begin{Prop}\label{charanal} Let $G$ be a non-compact, analytic $p$-adic group such that every proper, closed, characteristic subgroup is compact. Then one of the following (mutually exclusive) cases holds:
\begin{enumerate}
\item[(a)] $R(G)$ is compact open, central in $G$ and the discrete group $G/R(G)$ is infinitely generated and contains a non-abelian free subgroup.
\item[(b)] $G$ is isomorphic to $\Q_p^n$ for some $n>0$.
\item[(c)] $G\simeq S^k/Z$, where $S$ is $p$-adic of simple type, $k\ge 1$ an integer, and $Z$ a central subgroup of $S^k$, invariant under a transitive group of permutations of $\{1,\dots,k\}$.
\end{enumerate}
Conversely Cases (b) and (c) imply that any closed proper characteristic subgroup of $G$ is trivial in (b), finite central in (c).
\end{Prop}
We do not know if Case (a) can actually occur.

\begin{proof}
Suppose that $G$ has an open solvable subgroup. By Lemma \ref{normalsolv}, $G$ has a normal one, so $R(G)$ is open. We then have two cases
\begin{itemize}
\item $R(G)$ is compact open. Then $G$ cannot be virtually solvable, as otherwise $R(G)$ would have finite index and $G$ would be compact. By Tits' alternative, $G$ contains a free subgroup. Since $R(G)$ is solvable, this free subgroup maps injectively into the discrete group $G/R(G)$. Viewing $G\subset\GL(\Q_p)$, $G$ is contained in the normalizer $N$ of the Zariski-closure $R$ of $R(G)$. So $G/R(G)$ embeds injectively into $N/R$, which is contained in the group of $\Q_p$-points of a linear algebraic group; so if $G/R(G)$ is finitely generated, it is residually finite, and has a proper characteristic subgroup of finite index, a contradiction.
\item $G$ is solvable. Then $P=\overline{[G,G]}$ is a proper subgroup and is characteristic, so is compact. By Lemma \ref{comcen}, $P$ is central in $G$. Therefore, for all $g_0\in G$, the mapping $g\mapsto [g,g_0]$ defines a continuous homomorphism $G/P\to P$. But $G/P$ is by Lemma \ref{padicLie} an abelian group, either divisible or $p$-torsion, while $P$ is virtually isomorphic to $\Z_p^k$. Therefore this homomorphism has kernel of finite index in $G/P$. So either this homomorphism is trivial, for all $g_0$, and $G$ is abelian, and Lemma \ref{padicLie} allows to conclude that $G\simeq \Q_p^n$, or $G/P$ is a vector space over $\Z/p\Z$. This case actually cannot occur; indeed in this case the homomorphism above maps to the $p$-torsion in $P$, for each $g_0$. Therefore $[G,G]$ is contained in the $p$-torsion of $P$, so that $P$ is $p$-torsion. This implies that $G$ is of uniform torsion, which forces $G$ to be finite, contrary to our assumption.
\end{itemize}

Let us suppose now that $G$ has no open solvable subgroup, given as a subgroup of $\GL_n(\Q_p)$. Then its Lie algebra $\mathfrak{g}$ is not solvable (this follows for instance from Lemma \ref{radlie}). Let $\mathfrak{h}\neq\{0\}$ be the stable term of the derived series of $\mathfrak{g}$. Since $\mathfrak{h}$ is perfect, it is the Lie algebra of a unique connected Zariski-closed subgroup $H$. Since $\mathfrak{h}\subset\mathfrak{g}$, $G$ contains an open subgroup of $H$.
The subgroup $G\cap H$ of $G$ is characteristic in $G$, because any automorphism of $G$ stabilizes $\mathfrak{h}$. If by contradiction $G\cap H$ is compact, it is central by Lemma \ref{comcen}, but since $G$ contains a open subgroup of $H$, this would imply that the Lie algebra $\mathfrak{h}$ is abelian. This contradicts the fact that $\mathfrak{h}$ is perfect. So $G\subset H$, so that $\mathfrak{g}=\mathfrak{h}$.

Next, $R(H)\cap G$ is normal in $G$; since $G$ is not solvable, $R(H)\cap G$ is compact, hence central by Lemma \ref{comcen}. So $\mathfrak{g}=\mathfrak{r}\times\mathfrak{s}$, with $\mathfrak{s}$ semisimple, which is the Lie algebra of a unique connected semisimple subgroup $S$ of $H$. Since $\mathfrak{g}$ is perfect, we have $\mathfrak{r}=0$. Denote by $S_i$ be the isotypic factors of $S$ (each $S_i$ being the sums of simple factors with a given Lie algebra). Then $G\cap S_i$ is a closed characteristic subgroup of $G$, and contains an open subgroup of $S_i$. If it were compact, it would be central by Lemma \ref{comcen}, contradicting that $S_i$ is semisimple. So $G\cap S_i=G$ for all $i$. Of course this can happen only for one $i$. In other words, $S$ is isotypic and $G\subset S$. Since $G$ is not compact, $S$ has to be non-compact (isotropic); so $S^+$, the subgroup generated by unipotent elements, is a finite index subgroup of $S$ (see Theorem 2.3.1(c) in \cite{Mar}); necessarily $G\subset S^+$ and $G$ is an open non-compact subgroup; by the Howe-Moore property for $S^+$ and Proposition \ref{open}, $G=S^+$. Let $T$ be the universal covering of any simple factor of $S$. Then $G=T^k/Z$, with $Z$ a central subgroup, and the automorphisms of $G$ lift to automorphisms of $T^k$ preserving $Z$. Now any automorphism of $T$ permutes the $k$ copies of $T$, and by the condition on characteristic subgroups, the automorphisms of $T$ preserving $Z$ act transitively on the $k$ copies of $T$.
\end{proof}

\begin{Thm}\label{structureP-ADIC}
Let $G$ be an analytic $p$-adic group, $N$ a closed non-compact normal subgroup. The following are equivalent:
\begin{itemize}
\item[(i)] $(G,N)$ has the relative Howe-Moore Property;
\item[(ii)] One of the following properties holds:
\begin{enumerate}
\item[(1)] $N$ is isomorphic to $\Q_p^n$ for some $n>0$, and the action of any open subgroup of $G$ on $N$ is non-trivial and irreducible.
\item[(2)] $N\simeq S$, where $S$ is $p$-adic of simple type.
\end{enumerate}
\end{itemize}
\end{Thm}
\begin{proof}
Suppose $(G,N)$ has the relative Howe-Moore Property. By Proposition \ref{perm2}, every proper characteristic subgroup of $N$ is compact. So one of the three cases given by Proposition \ref{charanal} holds for $N$. We confront these cases to the conclusions of Theorem \ref{structurerelHM}.
\begin{itemize}
\item (Case (a) of Proposition \ref{charanal}) Suppose that the radical $R(N)$ of $N$ is compact open in $N$, and $N/R(N)$ is not finitely generated. We show that this case cannot happen. The radical $R(G)$ of $G$ is Zariski-closed in $G$, so the centralizer $Z_F$ of a subset $F$ modulo $R(G)$ is Zariski-closed as well. When $F$ is a growing finite subset of $N$, then $Z_F$ is a decreasing Zariski-closed subset of $G$. Therefore this stabilizes: there exists a finite subset $F$ of $N$ such that $Z_F=Z_N$. Now $Z_F$ is an open subgroup of $G$: indeed, as $N$ is normal in $G$, for $f\in F$ the map $c_f:G\rightarrow N:g\mapsto [g,f]$ continuous, and moreover $R(G)\cap N=R(N)$ is open in $N$ so that $Z_F=\bigcap_{f\in F}c_f^{-1}(R(G)\cap N)$ is open in $G$. Thus $Z_N$ is open in $G$.

    Consider then the pair $(G/R(G),N/(N\cap R(G)))$: it has the relative Howe-Moore property, by Proposition \ref{perm2}. Moreover $N/(N\cap R(G))$ is an infinite, discrete, normal subgroup in $G/R(G)$, whose centralizer is open (as the image of $Z_N$ in $G/R(G)$; so Proposition \ref{dino} applies to conclude that $N/(N\cap R(G))=N/R(N)$ is finitely generated, which is a contradiction.

\item (Case (b) of Proposition \ref{charanal}) $N\simeq\Q_p^n$ for some $n>0$ (this corresponds to case (1) in Theorem \ref{structurerelHM}). It follows from Propositions \ref{perm1}(ii) and \ref{perm2} that the action of $G$ on $N$ is irreducible in restriction to any open subgroup of $G$.
\item (Case (c) of Proposition \ref{charanal}) $N=S^k/Z$ with $S$ $p$-adic of simple type and $Z$ central. Then $N$ is compactly generated. By case (2.b) in Theorem \ref{structurerelHM}, we have $k=1$, i.e. $N$ is a $p$-adic group of simple type.

\end{itemize}
The converse is proved exactly as in Theorem \ref{structureLIEHM}.
\end{proof}

\end{document}